\documentclass[3p,10pt,nopreprintline]{elsarticle}

\usepackage{amssymb}

\usepackage{amsmath}
\usepackage{amsthm}
\usepackage{amssymb}
\usepackage{graphicx}
\usepackage{epsfig}
\usepackage[normal]{subfigure}
\usepackage{float}
\usepackage{paralist}
\usepackage{caption2}
\usepackage{hyperref}
\usepackage{aliascnt}
\usepackage[T1]{fontenc}
\usepackage{natbib}
\usepackage{CJK}
\usepackage{color}
\usepackage[table]{xcolor}
\usepackage{colortbl,booktabs}

\newtheoremstyle{theorem}{6pt}{6pt}{\itshape}{}{\bfseries}{.}{.5em}{}
\newtheoremstyle{definition}{6pt}{6pt}{\upshape}{}{\bfseries}{.}{.5em}{}

\theoremstyle{theorem}
\newtheorem{theorem}{Theorem}[section]

\allowdisplaybreaks[4]

\newaliascnt{corollary}{theorem}

\aliascntresetthe{corollary}

\newaliascnt{lemma}{theorem}
\newtheorem{lemma}[lemma]{Lemma}
\aliascntresetthe{lemma}

\newaliascnt{sublemma}{theorem}

\aliascntresetthe{sublemma}

\theoremstyle{definition}

\newtheorem{remark}{Remark}[section]

\newaliascnt{proposition}{theorem}
\newtheorem{proposition}[proposition]{Proposition}
\aliascntresetthe{proposition}

\newcommand{\R}{{\mathbb R}}

\newcommand{\dif}{{\mathrm d}}

\numberwithin{equation}{section}

\begin{document}

\begin{frontmatter}

\title{Global Strong Solutions to the Cauchy Problem of Three-dimensional Isentropic Magnetohydrodynamics Equations with Large Initial Data}

\author[label1]{Yachun Li}
\address[label1]{School of Mathematical Sciences, MOE-LSC, and SHL-MAC, Shanghai Jiao Tong University, Shanghai 200240, P.R.China;}
\ead{ycli@sjtu.edu.cn}

\author[label2]{Peng Lu}
\address[label2]{School of Mathematical Sciences, Shanghai Jiao Tong University, Shanghai 200240, P.R.China;}
\ead{lp95@sjtu.edu.cn}

\author[label3]{Zhaoyang Shang\corref{cor1}}
\address[label3]{School of Finance, Shanghai Lixin University of  Accounting and Finance, Shanghai 201209, P.R.China;}

\cortext[cor1]{Corresponding author.}
\ead{shangzhaoyang@sjtu.edu.cn}

\begin{abstract}
We consider the Cauchy problem to the three-dimensional isentropic compressible Magnetohydrodynamics (MHD) system with density-dependent viscosities. 
When the initial density is linearly equivalent to a large constant state, we prove that strong solutions exist globally in time, and there is  no restriction on the size of the initial velocity and initial magnetic field. As far as we know, this is the first result on the global well-posedness of density-dependent viscosities with large initial data for 3D compressible MHD equations.
\end{abstract}


\end{frontmatter}


\section{Introduction}
In this paper, we consider the following isentropic compressible  MHD equations:
\begin{equation}\label{1.1}
\begin{cases}
\displaystyle
\rho_t+\text{div}(\rho{{u}})=0 , \\[8pt]
\displaystyle
(\rho{{u}})_t+\text{div}(\rho{{u}}\otimes{{u}})+\nabla{P}=(\nabla\times{H})\times{H}+\text{div}\mathbb{T} ,  \\[8pt]
\displaystyle
{H}_t-\nabla{\times(u\times{H})=-\nabla\times(\nu\nabla\times{H}}),\quad \mathrm{div}{H}=0,
\end{cases} 
\end{equation}
where $x=(x_1,x_2,x_3)\in\R^3$ and $t\geq 0$ denote the spatial coordinate and time coordinate, $\rho$ denotes the density, $u=(u^1, u^2, u^3)$ the velocity, $H=(H^1, H^2, H^3)$ the magnetic field, $P(\rho)=\rho^\gamma$ the pressure, and $\nu$ the magnetic diffusion coefficient of the magnetic field. $\mathbb{T}$ is the viscosity stress tensor given by
$\mathbb{T}=2\mu\mathcal{D}(u)+\lambda\mathrm{div}u\mathbb{I}_3,$ where $\mathcal{D}(u)=\frac{1}{2}\left(\nabla u+(\nabla u)^\top\right)$ is the deformation tensor, $\mathbb{I}_3$ is the $3 \times 3$ identity matrix, $\mu$ is the shear viscosity coefficient, and $\lambda+\frac{2}{3}\mu$ is the bulk viscosity coefficient.

The compressible MHD equations (\ref{1.1}) describe the motion of a compressible viscous flow in the magnetic field. The issues of well-posedness and dynamical behaviors of MHD system have received wide attention in recent decades.
In three-dimensional space, the local existence of strong
solutions was obtained by Vol’pert-Hudjaev \cite{MR390528}, in 1972, for
the Cauchy problem with large initial data and the initial density being strictly positive. In 1984, the  first small smooth solution
was constructed by by Kawashima \cite{1984Systems} when the initial data are taken to be
close to a constant state in $H^3(\mathbb{R}^3)$. In 2010, Hu-Wang \cite{MR2646819} considered the initial-boundary value problem in a
bounded domain with large data and established the existence and large-time behavior of global
weak solutions for the adiabatic exponent $\gamma>\frac{3}{2}$. Later, in 2012, Suen-Hoff \cite{MR2927617} proved the global-in-time existence of weak solutions with initial data
small in $L^2(\mathbb{R}^3)$ and initial density positive and essentially bounded. Moreover, the authors concluded that small-energy smooth solutions described in \cite{1984Systems} are dense in the set of small-energy weak solutions. In 2013, Li-Xu-Zhang \cite{MR3056749}  considered the Cauchy problem for regular initial data with small energy but possibly large oscillations, and proved the global well-posedness of classical solution, where the density is allowed to contain vacuum states. The initial-boundary-value problem with small energy can be found in recent work \cite{MR4580966} given by Chen-Huang-Peng-Shi. For the results of well-posedness  in two-dimensional space,
we refer to \cite{MR0839315,MR3528824,MR3317636,MR3620698}, and the references cited therein.

Recently, Yu \cite{Yummas} considered the Cauchy problem to the 3D compressible Navier–Stokes system with the following density-dependent viscosities
\begin{equation}\label{viscosity}
\mu(\rho)=\bar{\mu}\rho^\alpha, \quad \lambda(\rho)=\bar{\lambda}\rho^\alpha,
\end{equation}
where
$\frac{4}{3}\leq\gamma\leq\alpha\leq\frac{5}{3}, \alpha+4\gamma>7,\alpha+\gamma \leq 3.$
When the initial density is sufficiently large, the author proved that strong solutions exist globally in time. 
Motivated by Yu's work \cite{Yummas}, 
in this paper we study the global existence of large strong solutions to the Cauchy problem (\ref{1.1}) with initial data 
\begin{align}\label{initial}
(\rho, u, H)(x,t)|_{t=0}=(\rho_0, u_0, H_0)(x), \quad x\in\mathbb{R}^3,
\end{align}
when magnetic diffusion $\nu$ is a constant, viscosities $\mu$ and $\lambda$ satisfy (\ref{viscosity}).
The far field behavior is given by
\begin{align}\label{far}
(\rho, u, H)(x,t)\to (\tilde{\rho}, 0, 0), \quad \text{as} \quad |x|\to\infty, ~ t>0.
\end{align}

Throughout this paper, for convenience, we denote
$$\displaystyle L^p=L^p(\mathbb{R}^3), ~ W^{k,p}=W^{k,p}(\mathbb{R}^3), ~ H^k=W^{k,2}, ~ D^{k,p}=\left\{f\in L^1_{loc}(\mathbb{R}^3) ~ | ~ \|\nabla^k u\|_{L^p}<\infty\right\},$$
and
$$\int f\dif x=\int_{\mathbb{R}^3}f\dif x$$
for any function $f$. Moreover, we use $C$ and $C_i$ ($i\in\mathbb{N}$) to denote a generic positive constant depending on $\nu$, $\bar{\lambda}$, $\bar{\mu}$, $\rho_0-\tilde{\rho}$, $u_0$ and $H_0$.\\

Our main theorem can be stated as follows.
\begin{theorem}\label{main Thm}
Let $\tilde{\rho}>1$ be a constant,  $1<\gamma\leq 3$, and $\alpha>\frac{3}{2}\gamma-\frac{1}{2}$. Suppose $(\rho_0,u_0,H_0)$ satisfy
\begin{align}\label{data}
\frac{3}{4}\tilde{\rho}\leq\rho_0(x)\leq\frac{5}{4}\tilde{\rho}, \quad \rho_0(x)-\tilde{\rho}\in H^1\cap D^{1,4}, \quad (u_0(x), H_0(x))\in H^2, \quad \mathrm{div}H_0=0,
\end{align}
and there exists a positive constant $M$ depending on $\gamma$, $\alpha$, $\bar{\lambda}$, $\bar{\mu}$, $\rho_0-\tilde{\rho}$, $u_0$ and $H_0$ such that $\tilde{\rho}\geq M$.
Then the Cauchy problem \eqref{1.1}, \eqref{initial}--\eqref{far} admits a global strong solution $(\rho, u, H)$ defined on $\mathbb{R}^3\times[0,\infty)$ satisfying
$$\frac{2}{3}\tilde{\rho}\leq\rho(x,t)\leq\frac{4}{3}\tilde{\rho}, \quad \text{for} \quad x\in\mathbb{R}^3, ~ t\geq 0,$$
and
$$
\begin{cases}
\rho-\tilde{\rho}\in C([0,\infty); H^1\cap D^{1,4}), \\
u\in C([0,\infty); H^2) \cap L^2(0,\infty; D^{2,4}), \quad H\in C([0,\infty); H^2), \\
(u_t, H_t)\in L^\infty(0,\infty; L^2)\cap L^2(0,\infty; H^1).
\end{cases}
$$
\end{theorem}

\begin{remark}
The main contribution of this paper is we establish the global existence of strong solutions to the Cauchy problem of 3D isentropic MHD equations with large initial data. In addition, the strong solutions are uniform in time in the above solution spaces and uniform upper bounds of the solution are given.
\end{remark}

\begin{remark}
When $H=0$ in (\ref{1.1}), our results in Theorem \ref{main Thm} reduce to the global well-posedness of 3D isentropic Navier-Stokes equations with large initial data. Compared with the results in \cite{Yummas}, in Theorem \ref{main Thm}, we refine the restrictions on the parameters of $\alpha$, $\gamma$ and there is no small assumptions on the initial kinetic energy
$\left\|\sqrt{\rho_0} u_{t 0}\right\|_{L^2}^2 \leq  C\tilde{\rho}$
 (see Lemma 3.4, \cite{Yummas}). More precisely, for isentropic Navier-Stokes equations, Theorem \ref{main Thm} can be proved provided $1<\gamma\leq 3$ and $\alpha>\frac{3}{8}\gamma+\frac{5}{8}$, where $\gamma\leq 3$ is used to give the sharp $L^2$-energy estimate in term of $\tilde{\rho}$. Moreover, when the viscosities are given in the form of  (\ref{viscosity}) and the density is large, the initial kinetic energy is equivalent to $\tilde{\rho}^{2\alpha-1}$, actually. For the general case of $\gamma>3$, we left the proof for the interested reader. When $H\neq0$, the effects of magnetic field need to be considered additionally, in this paper consider the constant magnetic diffusion and only use large dissipation property of viscosities to control the increase of magnetic field energy, where some restrictions on the interval of parameters $\alpha$ and $\gamma$ are required further, $\alpha>\frac{3}{2}\gamma-\frac{1}{2}$ for example. It is fortune for us to find that the interval of parameters is not empty set, whereas in \cite{Yummas} the effective interval length is only $\frac{1}{3}$. However, due to the interaction between the density and magnetic field, during the process of the first-order derivative estimates on the density, magnetic field plays an important role in our analysis. In order to overcome the key difficulties caused by the  magnetic field, some new uniform estimates are given in this paper. In fact, the  magnetic field can be controlled by a constant independent of $\tilde{\rho}$ in $H^1$ Sobolev space, which is different from the $\tilde{\rho}$-dependent estimates on the velocity and very important in our analysis. As far as we know, this is the first result on the uniform estimates of the density in $H^1$ Sobolev space and global strong solutions to the 3D compressible MHD equations with density-dependent viscosities and large initial data.
\end{remark}


\section{Uniform a priori estimates}
In this section, we establish some uniform a priori estimates for the local-in-time solution on $\R^3 \times (0,T]$. First, we take $A_i(T)$ as follows:
\begin{equation}\label{3.01}
A_1(T)=\left(\frac{\tilde{\rho}}{2}\right)^\alpha\sup_{0\leq t\leq T}\|\nabla{u}\|^2_{L^2}+\int_0^T\|\sqrt\rho{u_t}\|^2_{L^2}\dif t,\quad
A_2(T)=\sup_{0\leq t\leq T}\|\nabla{H}\|^2_{L^2}+\int_0^T\|{H_t}\|^2_{L^2}\dif t,
\end{equation}
\begin{equation}\label{3.02}
A_3(T)=\sup_{0\leq t\leq T}\left(\|\sqrt\rho u_t\|^2_{L^2}+\| H_t\|^2_{L^2}\right)+\int_0^T\left(\tilde{\rho}^\alpha\|\nabla u_t\|^2_{L^2}+\|\nabla H_t\|^2_{L^2}\right)\dif t.
\end{equation}
\begin{equation}\label{3.03}
A_4(T)=\sup_{0\leq t\leq T}\left(\|\nabla \rho\|_{L^2}^2+\|\nabla \rho\|_{L^4}^2\right)+\tilde{\rho}^{\gamma-\alpha}\int_0^T\left(\|\nabla \rho\|_{L^2}^2+\|\nabla \rho\|_{L^4}^2\right)\dif t.
\end{equation}

\begin{proposition}\label{p}
Assume that the solution $(\rho, u, H)$ satisfy
\begin{equation}\label{3.09}
A_1(T)\leq 3\tilde{\rho}^{\alpha+\frac{3}{4}\gamma-\frac{3}{4}+2\delta}, \; A_2(T)\leq 3\|\nabla H_0\|^2_{L^2}, \; A_3(T)\leq 3\tilde{\rho}^{2\alpha-1+\delta}, \; A_4(T)\leq 3\tilde{\rho}^{2+\delta}, \; \|\rho-\tilde{\rho}\|_{L^\infty}\leq\frac{\tilde{\rho}}{2},
\end{equation}
for any $(x,t)\in {{\mathbb{R}}}^3\times[0,T]$, then
\begin{equation}\label{3.011}
A_1(T)\leq 2\tilde{\rho}^{\alpha+\frac{3}{4}\gamma-\frac{3}{4}+2\delta}, \; A_2(T)\leq 2\|\nabla H_0\|^2_{L^2}, \; A_3(T)\leq 2\tilde{\rho}^{2\alpha-1+\delta}, \; A_4(T)\leq 2\tilde{\rho}^{2+\delta}, \; \|\rho-\tilde{\rho}\|_{L^\infty}\leq\frac{\tilde{\rho}}{3},
\end{equation}
provided $\tilde{\rho}\geq M$, where $1<\gamma\leq3$, $\alpha>\frac{3}{2}\gamma-\frac{1}{2}$, $0<\delta<\frac{1}{16}\min\{\gamma-1, ~ 2\alpha+1-3\gamma\}$.
\end{proposition}

\begin{lemma}\label{lem:2.1}
Assume that \eqref{3.09} holds, then there exists a positive constant $M_1$, such that 
\begin{equation}\label{3.1}
\sup_{t\in[0,T]}\left(\|G(\rho)\|_{L^1}+\tilde{\rho}\|u\|_{L^2}^2\right)+{\tilde{\rho}}^\alpha\int_0^T\|\nabla u\|_{L^2}^2\dif t\leq C\tilde{\rho},
\end{equation}
\begin{equation}\label{3.1-1}
\sup_{t\in[0,T]}\|H\|_{L^2}^2+\nu\int_0^T\|\nabla H\|_{L^2}^2\dif t\leq 2\|H_0\|_{L^2}^2,
\end{equation}
provided $\tilde{\rho}\geq M_1$, where  $\displaystyle G(\rho)  \triangleq \rho\int_{\tilde{\rho}}^\rho \frac{P(s)-P(\tilde{\rho})}{s^2}\dif s.$
\end{lemma}
\begin{proof}
First, we multiply $(\ref{1.1})_1$, $(\ref{1.1})_2$ by $G^\prime(\rho)$ and $u$, respectively, and integrate over ${{\mathbb{R}}}^3$, after integration by parts, we obtain
\begin{align}\label{3.2}
&\frac{\dif}{\dif t}\left(\|G(\rho)\|_{L^1}+\frac{1}{2}\|\sqrt{\rho}u\|_{L^2}^2+\frac{1}{2}\|H\|_{L^2}^2\right)+\left(\frac{\tilde{\rho}}{2}\right)^\alpha\|\nabla u\|_{L^2}^2+\|\nabla H\|_{L^2}^2\nonumber\\
\leq~&\frac{\dif}{\dif t}\left(\|G(\rho)\|_{L^1}+\frac{1}{2}\|\sqrt{\rho}u\|_{L^2}^2+\frac{1}{2}\|H\|_{L^2}^2\right)+\int\left(2\rho^\alpha |\mathcal{D}(u)|^2 +\rho^\alpha(\text{div}u)^2 \right)\dif x+\|\nabla H\|_{L^2}^2\nonumber\\
\leq~& 0.
\end{align}
Integrating the above inequality over $(0,T)$, we have
\begin{align}\label{3.3}
\sup_{t\in[0,T]}\left(\|G(\rho)\|_{L^1}+\tilde{\rho}\|u\|_{L^2}^2\right)+{\tilde{\rho}}^\alpha\int_0^T\|\nabla u\|_{L^2}^2\dif t
\leq \frac{C}{\tilde{\rho}^{2-\gamma}}\|\rho_0-\tilde{\rho}\|_{L^2}^2+C\tilde{\rho}\leq C\tilde{\rho},
\end{align}
where we have used $G(\rho_0)\sim \tilde{\rho}^{\gamma-2}(\rho_0-\tilde{\rho})^2$, $\gamma\leq 3$, and assumption \eqref{3.09}. Then,  (\ref{3.1}) is proved.

Next, multiplying (\ref{1.1})$_3$ by $H$ and integrating over $\mathbb{R}^3$, we have
\begin{align}\label{3.3-1}
\frac{\dif}{\dif t}\|H\|_{L^2}^2+\nu\|\nabla H\|_{L^2}^2
\leq~& \int |\nabla u||H|^2 \dif x+\int |u||\nabla H||H| \dif x\nonumber\\
\leq~&  \int|\nabla u|^2|H|^2 \dif x+ C\int|u|^2|\nabla H|^2 \dif x\nonumber\\
\leq~& \|\nabla u\|_{L^2}\|H\|_{L^4}^2+ \|u\|_{L^4}\|\nabla H\|_{L^2}\|H\|_{L^4}\nonumber\\
\leq~& C\|\nabla u\|_{L^2}\|H\|_{L^2}^\frac{1}{2}\|\nabla H\|_{L^2}^\frac{3}{2}+C\|u\|_{L^2}^\frac{1}{4}\|\nabla u\|_{L^2}^\frac{3}{4}\|H\|_{L^2}^\frac{1}{4}\|\nabla H\|_{L^2}^\frac{7}{4}\nonumber\\
\leq~& \frac{\nu}{4}\|\nabla H\|_{L^2}^2+ C\|\nabla u\|_{L^2}^4\|H\|_{L^2}^2+C\|u\|_{L^2}^2\|\nabla u\|_{L^2}^6\|H\|_{L^2}^2\nonumber\\
\leq~& \frac{\nu}{4}\|\nabla H\|_{L^2}^2+ C\left(\tilde{\rho}^{\frac{3}{4}\gamma-\frac{3}{4}+2\delta}+\tilde{\rho}^{\frac{3}{2}\gamma-\frac{3}{2}+4\delta}\right)\|\nabla u\|_{L^2}^2\|H\|_{L^2}^2,
\end{align}
which together with Gronwall's inequality gives
\begin{align}\label{3.3-2}
\|H\|_{L^2}^2+\nu\int_{0}^{T}\|\nabla H\|_{L^2}^2\dif t
\leq~&\|H_0\|_{L^2}^2\exp\left\{C\left(\tilde{\rho}^{\frac{3}{4}\gamma-\frac{3}{4}+2\delta}+\tilde{\rho}^{\frac{3}{2}\gamma-\frac{3}{2}+4\delta}\right)\int_{0}^{T}\|\nabla u\|_{L^2}^2\dif t\right\}\nonumber\\
\leq~&\|H_0\|_{L^2}^2\exp\left\{C_1\tilde{\rho}^{\frac{3}{2}\gamma-\frac{1}{2}+4\delta-\alpha}\right\} \nonumber\\
\leq~& 2\|H_0\|_{L^2}^2,
\end{align}
provided that $\alpha>\frac{3}{2}\gamma-\frac{1}{2}$ and $\tilde{\rho}\geq M_1\triangleq\left(\frac{C_1}{\ln 2}\right)^{\frac{2}{2\alpha+1-3\gamma-8\delta}}$, and thus we complete the proof.
\end{proof}

\begin{lemma}\label{lem:3.3}
Assume that \eqref{3.09} holds, then we have 
\begin{equation}\label{3.5}
\sup_{t\in[0,T]}\|\nabla^2 u\|_{L^2}^2\leq C\tilde{\rho}^{\frac{3}{4}\gamma-\frac{3}{4}+6\delta},\quad \int_0^T\|\nabla^2 u\|_{L^2}^2\dif t\leq C\tilde{\rho}^{\gamma-\alpha+\delta}.
\end{equation}
\begin{equation}\label{3.6}
\|\nabla^2 H\|_{L^2}^2\leq C\tilde{\rho}^{2\alpha-1+\delta}, \quad \int_0^T\|\nabla^2 H\|_{L^2}^2\dif t\leq C.
\end{equation}

\begin{proof}
Denote $\mathcal{P}(\rho)=\frac{A\gamma}{\gamma-\alpha}\rho^{\gamma-\alpha}$. We rewrite (\ref{1.1})$_2$ as
\begin{equation*}\label{3.7}
-\Delta u-2\nabla \text{div} u+ \nabla \mathcal{P}(\rho)=\rho^{-\alpha}\left(-\rho u_t -\rho u\cdot \nabla u+ 2\nabla \rho^\alpha\cdot \mathcal {D}(u)+\nabla \rho^\alpha\cdot \text{div} u\mathbb{I}_3+H\cdot\nabla H-\frac{1}{2}\nabla |H|^2\right).
\end{equation*}

According to the standard $L^2$-estimate of elliptic system, Young's inequality, and 
\begin{align}\label{3.8}
\|\nabla^2 u\|_{L^2}&\leq C\|\rho^{-\alpha}\|_{L^\infty}\left(\|\rho u_t\|_{L^2}+\|\rho u\cdot \nabla u\|_{L^2}+\|\rho^{\alpha-1}\nabla\rho\nabla u\|_{L^2}+\|H\cdot\nabla H\|_{L^2}+\| \nabla \mathcal{P}(\rho)\|_{L^2}\right)\nonumber\\
&\leq  \frac{C}{\tilde{\rho}^{\alpha}}\left(\tilde{\rho}^{\frac{1}{2}}\|\sqrt{\rho} u_t\|_{L^2}+\tilde{\rho}\|u\cdot \nabla u\|_{L^2}+\tilde{\rho}^{\alpha-1}\|\nabla\rho\nabla u\|_{L^2}+\|H\cdot\nabla H\|_{L^2}+\tilde{\rho}^{\gamma-1}\| \nabla \rho\|_{L^2}\right)\nonumber\\
 &\leq  C\bigg(\frac{1}{\tilde{\rho}^{\alpha-\frac{1}{2}}}\|\sqrt{\rho} u_t\|_{L^2}+\frac{1}{\tilde{\rho}^{\alpha-1}}\|u\|_{L^6}\|\nabla u\|_{L^3}+\frac{1}{\tilde{\rho}}\|\nabla\rho\|_{L^4}\|\nabla u\|_{L^4}\nonumber\\
 &\quad\quad+\frac{1}{\tilde{\rho}^{\alpha}}\|H\|_{L^6}\|\nabla H\|_{L^3}+\frac{1}{\tilde{\rho}^{\alpha-\gamma+1}}\| \nabla \rho\|_{L^2}\bigg)\nonumber\\
&\leq  C\bigg(\frac{1}{\tilde{\rho}^{\alpha-\frac{1}{2}}}\|\sqrt{\rho} u_t\|_{L^2}+\frac{1}{\tilde{\rho}^{\alpha-1}}\|\nabla u\|_{L^2}^\frac{3}{2}\|\nabla^2 u\|_{L^2}^\frac{1}{2}+\frac{1}{\tilde{\rho}}\|\nabla\rho\|_{L^4}\|\nabla u\|_{L^2}^\frac{1}{4}\|\nabla^2 u\|_{L^2}^\frac{3}{4}\nonumber\\
&\quad\quad+\frac{1}{\tilde{\rho}^{\alpha}}\|\nabla H\|_{L^2}^\frac{3}{2}\|\nabla^2 H\|_{L^2}^\frac{1}{2}+\frac{1}{\tilde{\rho}^{\alpha-\gamma+1}}\| \nabla \rho\|_{L^2}\bigg)\nonumber\\
&\leq  \frac{1}{4}\|\nabla^2 u\|_{L^2}+C\bigg(\frac{1}{\tilde{\rho}^{\alpha-\frac{1}{2}}}\|\sqrt{\rho} u_t\|_{L^2}+\frac{1}{\tilde{\rho}^{2\alpha-2}}\|\nabla u\|_{L^2}^3+\frac{1}{\tilde{\rho}^4}\|\nabla\rho\|_{L^4}^4\|\nabla u\|_{L^2}\nonumber\\
&\quad\quad+\frac{1}{\tilde{\rho}^{\alpha}}\|\nabla H\|_{L^2}^\frac{3}{2}\big(\|H_t\|_{L^2}^\frac{1}{2}+\|\nabla u\|_{L^2}\|\nabla H\|_{L^2}^\frac{1}{2}+\|\nabla H\|_{L^2}^\frac{1}{2}\|\nabla u\|_{L^2}^\frac{1}{4}\|\nabla^2 u\|_{L^2}^\frac{1}{4}\big)+\frac{1}{\tilde{\rho}^{\alpha-\gamma+1}}\| \nabla \rho\|_{L^2}\bigg)\nonumber\\
&\leq \frac{1}{4}\|\nabla^2 u\|_{L^2}+C\bigg(\frac{1}{\tilde{\rho}^{\alpha-\frac{1}{2}}}\|\sqrt{\rho} u_t\|_{L^2}+\frac{1}{\tilde{\rho}^{2\alpha-\frac{3}{4}\gamma-\frac{5}{4}-2\delta}}\|\nabla u\|_{L^2}+\tilde{\rho}^{2\delta}\|\nabla u\|_{L^2}+\frac{1}{\tilde{\rho}^{\alpha}}\|\nabla H\|_{L^2}^\frac{1}{2}\|H_t\|_{L^2}^\frac{1}{2}\nonumber\\
&\quad\quad+\frac{1}{\tilde{\rho}^{\alpha}}\|\nabla u\|_{L^2}+\frac{1}{\tilde{\rho}^{\frac{4}{3}\alpha}}\|\nabla H\|_{L^2}^\frac{2}{3}\|\nabla u\|_{L^2}^\frac{1}{3}+\frac{1}{\tilde{\rho}^{\alpha-\gamma+1}}\| \nabla \rho\|_{L^2}\bigg),
\end{align}
where we have used (\ref{1.1})$_3$ and 
\begin{align}\label{3.9}
\|\nabla^2 H\|_{L^2}&\leq C\left(\|H_t\|_{L^2}+\|u\cdot \nabla H\|_{L^2}+\|\nabla u\cdot  H\|_{L^2}\right)\nonumber\\
&\leq C\left(\|H_t\|_{L^2}+\|u\|_{L^6}\|\nabla H\|_{L^3}+\|H\|_{L^6}\|\nabla u\|_{L^3}\right)\nonumber\\
&\leq C\big(\|H_t\|_{L^2}+\|\nabla u\|_{L^2}\|\nabla H\|_{L^2}^\frac{1}{2}\|\nabla^2 H\|_{L^2}^\frac{1}{2}+\|\nabla H\|_{L^2}\|\nabla u\|_{L^2}^\frac{1}{2}\|\nabla^2 u\|_{L^2}^\frac{1}{2}\big)\nonumber\\
&\leq \frac{1}{4}\|\nabla^2 H\|_{L^2}+ C\big(\|H_t\|_{L^2}+\|\nabla u\|_{L^2}^2\|\nabla H\|_{L^2}+\|\nabla H\|_{L^2}\|\nabla u\|_{L^2}^\frac{1}{2}\|\nabla^2 u\|_{L^2}^\frac{1}{2}\big).
\end{align}
Then, from (\ref{3.8}), we obtain
\begin{align}\label{3.11}
\|\nabla^2 u\|_{L^2}^2
\leq~& \frac{C}{\tilde{\rho}^{2\alpha-1}}\|\sqrt{\rho} u_t\|_{L^2}^2+C\left(\frac{1}{\tilde{\rho}^{4\alpha-\frac{3}{2}\gamma-\frac{5}{2}-4\delta}}+\tilde{\rho}^{4\delta}+\frac{1}{\tilde{\rho}^{2\alpha}}\right)\|\nabla u\|_{L^2}^2+\frac{C}{\tilde{\rho}^{2\alpha}}\|\nabla H\|_{L^2}\|H_t\|_{L^2}\nonumber\\
&+\frac{C}{\tilde{\rho}^{\frac{8}{3}\alpha}}\|\nabla H\|_{L^2}^\frac{4}{3}\|\nabla u\|_{L^2}^\frac{2}{3}+\frac{C}{\tilde{\rho}^{2(\alpha-\gamma+1)}}\| \nabla \rho\|_{L^2}^2\nonumber\\
\leq~& C\bigg(\tilde{\rho}^{\delta}+\frac{1}{\tilde{\rho}^{4\alpha-\frac{7}{4}-6\delta-\frac{9}{4}\gamma}}+\tilde{\rho}^{\frac{3}{4}\gamma-\frac{3}{4}+6\delta}+\frac{1}{\tilde{\rho}^{2\alpha-\frac{3}{4}\gamma+\frac{3}{4}-2\delta}}+\frac{1}{\tilde{\rho}^{\alpha+\frac{1}{2}-\frac{1}{2}\delta}}+\frac{1}{\tilde{\rho}^{\frac{8}{3}\alpha-\frac{1}{4}(\gamma-1)-\frac{2}{3}\delta}}+\frac{1}{\tilde{\rho}^{2(\alpha-\gamma)-\delta}}\bigg)\nonumber\\
\leq~& C\tilde{\rho}^{\frac{3}{4}\gamma-\frac{3}{4}+6\delta},
\end{align}
provided that $\alpha\geq\frac{3}{8}\gamma+\frac{5}{8}$. Integrating \eqref{3.11} over $(0,T)$, we have
\begin{align*}
\int_0^T\|\nabla^2 u\|_{L^2}^2\dif t
\leq~& C\bigg(\frac{1}{\tilde{\rho}^{2\alpha-1}}\int_0^T\|\sqrt{\rho} u_t\|_{L^2}^2\dif t+\left(\frac{1}{\tilde{\rho}^{4\alpha-\frac{3}{2}\gamma-\frac{5}{2}-4\delta}}+\tilde{\rho}^{4\delta}+\frac{1}{\tilde{\rho}^{2\alpha}}\right)\int_0^T\|\nabla u\|_{L^2}^2\dif t\nonumber\\
&+\frac{1}{\tilde{\rho}^{2\alpha}}\int_0^T\|\nabla H\|_{L^2}\|H_t\|_{L^2}\dif t+\frac{1}{\tilde{\rho}^{\frac{8}{3}\alpha}}\int_0^T\|\nabla H\|_{L^2}^\frac{4}{3}\|\nabla u\|_{L^2}^\frac{2}{3}\dif t+\frac{1}{\tilde{\rho}^{2(\alpha-\gamma+1)}}\int_0^T\| \nabla \rho\|_{L^2}^2\dif t\bigg)\nonumber\\
\leq~& C\bigg(\frac{1}{\tilde{\rho}^{\alpha-\frac{1}{4}-\frac{3}{4}\gamma-2\delta}}+\frac{1}{\tilde{\rho}^{5\alpha-\frac{3}{2}\gamma-\frac{7}{2}-4\delta}}+\frac{1}{\tilde{\rho}^{\alpha-1-4\delta}}+\frac{1}{\tilde{\rho}^{3\alpha-1}}+\frac{1}{\tilde{\rho}^{2\alpha-\frac{1}{2}}}+\frac{1}{\tilde{\rho}^{\frac{7}{3}\alpha-1}}+\frac{1}{\tilde{\rho}^{\alpha-\gamma-\delta}}\bigg)\nonumber\\
\leq~& C\tilde{\rho}^{\gamma-\alpha+\delta},
\end{align*}
provided that $\alpha>\frac{1}{8}\gamma+\frac{7}{8}$. 

Moreover, after direct calculations, $\|\nabla^2 H\|_{L^2}^2\leq C\tilde{\rho}^{2\alpha-1+\delta}
$ can be proved from (\ref{3.9}), since $\| H_t\|^2_{L^2}$ is the leading term. At last,
\begin{align}\label{3.11:2}
\int_0^T\|\nabla^2 H\|_{L^2}^2\dif t
&\leq C\int_0^T\big(\|H_t\|_{L^2}^2+\|\nabla u\|_{L^2}^4\|\nabla H\|_{L^2}^2+\|\nabla H\|_{L^2}^2\|\nabla u\|_{L^2}\|\nabla^2 u\|_{L^2}\big)\dif t\nonumber\\
&\leq C+C\tilde{\rho}^{\frac{3}{4}\gamma-\frac{3}{4}+2\delta}\int_0^T\|\nabla u\|_{L^2}^2\dif t+C\int_0^T\|\nabla u\|_{L^2}\|\nabla^2 u\|_{L^2}\dif t\nonumber\\
&\leq C+C\tilde{\rho}^{\frac{3}{4}\gamma+\frac{1}{4}+2\delta-\alpha}+C\tilde{\rho}^{\frac{1}{2}+\frac{1}{2}\gamma-\alpha+\frac{1}{2}\delta}\nonumber\\
&\leq C,
\end{align}
provided that $\alpha>\frac{3}{4}\gamma+\frac{1}{4}$. 
This completes the proof of Lemma \ref{lem:3.3}.
\end{proof}
\end{lemma}

\begin{lemma}\label{lem:3.4}
Assume that \eqref{3.09} holds, then we have 
\begin{equation}\label{3.12}
\left(\frac{\tilde{\rho}}{2}\right)^\alpha\sup_{0\leq t\leq T}\|\nabla{u}\|^2_{L^2}+\int_0^T\|\sqrt\rho{u_t}\|^2_{L^2}\dif t\leq 2\tilde{\rho}^{\alpha+\frac{3}{4}\gamma-\frac{3}{4}+2\delta},
\end{equation}

\begin{equation}\label{3.13}
\sup_{0\leq t\leq T}\|\nabla{H}\|^2_{L^2}+\int_0^T\|{H_t}\|^2_{L^2}\dif t\leq 2\|\nabla H_0\|^2_{L^2},
\end{equation}
provided there exists a constant  $M_2$, such that $\tilde{\rho}\geq M_2$.
\begin{proof}
Multiplying (\ref{1.1})$_2$ by $u_t$, integrating over ${{\mathbb{R}}}^3$, and after integration by parts, we have
\begin{align}\label{3.13.5}
&\frac{\dif}{\dif t}\int\left(\rho^\alpha|\mathcal{D}(u)|^2+\frac{\rho^\alpha}{2}|\text{div}u|^2\right)\dif x+\int\rho |u_t|^2\dif x\nonumber\\
=~&\frac{\dif}{\dif t}\int\left(P(\rho)-P( \tilde{\rho})\right)\text{div} u\dif x-\int P_t(\rho)\text{div} u\dif x +\int\left(\left(\rho^\alpha\right)_t|\mathcal {D}(u)|^2+\frac{1}{2}\left(\rho^\alpha\right)_t|\text{div} u|^2\right)\dif x\nonumber\\
&+\int\left(-\rho u\cdot \nabla u+ 2\nabla \rho^\alpha\cdot \mathcal {D}(u)+\nabla \rho^\alpha\cdot \text{div} u\mathbb{I}_3+H\cdot\nabla H-\frac{1}{2}\nabla |H|^2\right)\cdot u_t\dif x\nonumber\\
=~& \frac{\dif}{\dif t}\int\left(P(\rho)-P( \tilde{\rho})\right)\text{div} u\dif x+I_1+I_2+I_3,
\end{align}
where
\begin{align}\label{3.14}
I_1&=-\int P_t(\rho)\text{div} u \dif x\nonumber\\
&\leq C\left( \tilde{\rho}^{\gamma-1}\|\nabla\rho\|_{L^2}\|u\|_{L^\infty}\|\nabla u\|_{L^2}+\tilde{\rho}^{\gamma}\|\nabla u\|_{L^2}^2\right)\nonumber\\
&\leq C\left( \tilde{\rho}^{\gamma-1}\|\nabla\rho\|_{L^2}\|\nabla u\|_{L^2}^\frac{3}{2}\|\nabla^2 u\|_{L^2}^\frac{1}{2}+\tilde{\rho}^{\gamma}\|\nabla u\|_{L^2}^2\right)\nonumber\\
&\leq C\left(\tilde{\rho}^{\gamma+\frac{1}{2}\delta}\|\nabla u\|_{L^2}^\frac{3}{2}\|\nabla^2 u\|_{L^2}^\frac{1}{2}+\tilde{\rho}^{\gamma}\|\nabla u\|_{L^2}^2\right),
\end{align}

\begin{align}\label{3.15}
I_2&=\int\left(\left(\rho^\alpha\right)_t|\mathcal {D}(u)|^2+\frac{1}{2}\left(\rho^\alpha\right)_t|\text{div} u|^2\right)\dif x\nonumber\\
&\leq C\left( \tilde{\rho}^{\alpha-1}\|\nabla\rho\|_{L^4}\|u\|_{L^\infty}\|\nabla u\|_{L^4}\|\nabla u\|_{L^2}+\tilde{\rho}^{\alpha}\|\nabla u\|_{L^3}^3\right)\nonumber\\
&\leq C\left( \tilde{\rho}^{\alpha-1}\|\nabla\rho\|_{L^4}\|\nabla u\|_{L^2}^\frac{7}{4}\|\nabla^2 u\|_{L^2}^\frac{5}{4}+\tilde{\rho}^{\alpha}\|\nabla u\|_{L^2}^\frac{3}{2}\|\nabla^2 u\|_{L^2}^\frac{3}{2}\right)\nonumber\\
&\leq C\left( \tilde{\rho}^{\alpha+\frac{3}{2}\delta+\frac{3}{8}(\gamma-1)}\|\nabla u\|_{L^2}^\frac{3}{4}\|\nabla^2 u\|_{L^2}^\frac{5}{4}+\tilde{\rho}^{\alpha+\delta+\frac{3}{8}(\gamma-1)}\|\nabla u\|_{L^2}^\frac{1}{2}\|\nabla^2 u\|_{L^2}^\frac{3}{2}\right),
\end{align}

\begin{align}\label{3.16}
I_3=~&\int\left(-\rho u\cdot \nabla u+ 2\nabla \rho^\alpha\cdot \mathcal {D}(u)+\nabla \rho^\alpha\cdot \text{div} u\mathbb{I}_3+H\cdot\nabla H-\frac{1}{2}\nabla |H|^2\right)\cdot u_t \dif x\nonumber\\
\leq~& C\left(\tilde{\rho}^\frac{1}{2}\|\sqrt{\rho} u_t\|_{L^2}\|u\|_{L^\infty}\|\nabla u\|_{L^2}+\tilde{\rho}^{\alpha-\frac{3}{2}}\|\sqrt{\rho} u_t\|_{L^2}\|\nabla \rho\|_{L^4}\|\nabla u\|_{L^4}+\tilde{\rho}^{-\frac{1}{2}}\|\sqrt{\rho} u_t\|_{L^2}\|H\|_{L^\infty}\|\nabla H\|_{L^2}\right)\nonumber\\
\leq~& C\bigg(\tilde{\rho}^\frac{1}{2}\|\sqrt{\rho} u_t\|_{L^2}\|\nabla u\|_{L^2}^\frac{3}{2}\|\nabla^2 u\|_{L^2}^\frac{1}{2}+\tilde{\rho}^{\alpha-\frac{3}{2}}\|\sqrt{\rho} u_t\|_{L^2}\|\nabla \rho\|_{L^4}\|\nabla u\|_{L^2}^\frac{1}{4}\|\nabla^2 u\|_{L^2}^\frac{3}{4}\nonumber\\
&+\tilde{\rho}^{-\frac{1}{2}}\|\sqrt{\rho} u_t\|_{L^2}\|\nabla H\|_{L^2}^\frac{3}{2}\|\nabla^2 H\|_{L^2}^\frac{1}{2}\bigg)\nonumber\\
\leq~& C\bigg(\tilde{\rho}^{\frac{1}{8}+\frac{3}{8}\gamma+\delta}\|\sqrt{\rho} u_t\|_{L^2}\|\nabla u\|_{L^2}^\frac{1}{2}\|\nabla^2 u\|_{L^2}^\frac{1}{2}+\tilde{\rho}^{\alpha-\frac{1}{2}+\frac{1}{2}\delta}\|\sqrt{\rho} u_t\|_{L^2}\|\nabla u\|_{L^2}^\frac{1}{4}\|\nabla^2 u\|_{L^2}^\frac{3}{4}\nonumber\\
&+\tilde{\rho}^{-\frac{1}{2}}\|\sqrt{\rho} u_t\|_{L^2}\|\nabla H\|_{L^2}^\frac{1}{2}\|\nabla^2 H\|_{L^2}^\frac{1}{2}\bigg)\nonumber\\
\leq~&\frac{1}{2}\|\sqrt{\rho} u_t\|_{L^2}^2+ C\left(\tilde{\rho}^{\frac{1}{4}+\frac{3}{4}\gamma+2\delta}\|\nabla u\|_{L^2}\|\nabla^2 u\|_{L^2}+\tilde{\rho}^{2\alpha-1+\delta}\|\nabla u\|_{L^2}^\frac{1}{2}\|\nabla^2 u\|_{L^2}^\frac{3}{2}+\tilde{\rho}^{-1}\|\nabla H\|_{L^2}\|\nabla^2 H\|_{L^2}\right).
\end{align}

Inserting (\ref{3.14})--(\ref{3.16}) into (\ref{3.13}), and integrating with respect to time over $(0,T)$, by using assumption \eqref{3.09}, Lemma \ref{lem:2.1} and Lemma \ref{lem:3.3}, we obtain
\begin{align}\label{3.17}
&\left(\frac{\tilde{\rho}}{2}\right)^\alpha\sup_{0\leq t\leq T}\|\nabla{u}\|^2_{L^2}+\frac{1}{2}\int_0^T\|\sqrt{\rho} u_t\|_{L^2}^2\dif t \nonumber\\
\leq~& (2\tilde{\rho})^\alpha C_0 +\|P(\rho)-P( \tilde{\rho})\|_{L^2}\|\nabla u\|_{L^2}+\|P(\rho_0)-P( \tilde{\rho})\|_{L^2}\|\nabla u_0\|_{L^2}\nonumber\\
&+C\int_0^T\bigg(\tilde{\rho}^{\gamma+\frac{1}{2}\delta}\|\nabla u\|_{L^2}^\frac{3}{2}\|\nabla^2 u\|_{L^2}^\frac{1}{2}+\tilde{\rho}^{\gamma}\|\nabla u\|_{L^2}^2+\tilde{\rho}^{\alpha+\frac{3}{16}(\gamma-1)+\delta}\|\nabla u\|_{L^2}^\frac{3}{4}\|\nabla^2 u\|_{L^2}^\frac{5}{4}\nonumber\\
&+\tilde{\rho}^{\alpha+\frac{\delta}{2}+\frac{3}{16}(\gamma-1)}\|\nabla u\|_{L^2}^\frac{1}{2}\|\nabla^2 u\|_{L^2}^\frac{3}{2}+\tilde{\rho}^{\frac{1}{4}+\frac{3}{4}\gamma+2\delta}\|\nabla u\|_{L^2}\|\nabla^2 u\|_{L^2}+\tilde{\rho}^{2\alpha-1+\delta}\|\nabla u\|_{L^2}^\frac{1}{2}\|\nabla^2 u\|_{L^2}^\frac{3}{2}\nonumber\\
&+\tilde{\rho}^{-1}\|\nabla H\|_{L^2}\|H_t\|_{L^2}+\tilde{\rho}^{-1}\|\nabla u\|_{L^2}^2\|\nabla H\|_{L^2}^2+\tilde{\rho}^{-1}\|\nabla u\|_{L^2}^\frac{1}{2}\|\nabla^2 u\|_{L^2}^\frac{1}{2}\|\nabla H\|_{L^2}^2\bigg)\dif t\nonumber\\
\leq~& (2\tilde{\rho})^\alpha C_0 +C\tilde{\rho}^{\frac{5+11\gamma}{16}+\delta}+C\bigg(\tilde{\rho}^{\frac{5}{4}\gamma+\frac{3}{4}\delta+\frac{3}{4}-\alpha}+\tilde{\rho}^{\gamma+1-\alpha}+\tilde{\rho}^{\frac{3}{16}+\frac{13}{16}\gamma+\frac{13}{8}\delta}+\tilde{\rho}^{\frac{1}{16}+\frac{15}{16}\gamma+\frac{5}{4}\delta}+\tilde{\rho}^{\frac{3}{4}-\alpha+\frac{5}{4}\gamma+\frac{5}{2}\delta}\nonumber\\
&+\tilde{\rho}^{\alpha-\frac{3}{4}+\frac{3}{4}\gamma+\frac{7}{4}\delta}+\tilde{\rho}^{-1}+\tilde{\rho}^{-\alpha}+\tilde{\rho}^{\frac{\gamma}{4}-\frac{\alpha}{2}-\frac{3}{4}+\frac{\delta}{4}}\bigg)\nonumber\\
\leq~& 2^\alpha C_0 \tilde{\rho}^\alpha + \tilde{\rho}^{\alpha+\frac{3}{4}\gamma-\frac{3}{4}+2\delta}\nonumber\\
\leq~&  2\tilde{\rho}^{\alpha+\frac{3}{4}\gamma-\frac{3}{4}+2\delta},
\end{align}
provided that $\alpha\geq\frac{1}{4}\gamma+\frac{3}{4}$ and $\tilde{\rho}\geq M_{21}\triangleq(2^\alpha C_0)^{\frac{4}{3(\alpha-1)+8\delta}}$.

At last, multiplying (\ref{1.1})$_3$ by $H_t$ and integrating by parts, we have
\begin{align}\label{3.18}
\frac{\dif}{\dif t}\|\nabla H\|_{L^2}^2+\|H_t\|_{L^2}^2
&\leq C\int |\nabla u||H||H_t| \dif x+C\int |u||\nabla H||H_t| \dif x\nonumber\\
&\leq \frac{1}{4}\int |H_t|^2 \dif x+ C\int|\nabla u|^2|H|^2 \dif x+ C\int|u|^2|\nabla H|^2 \dif x\nonumber\\
&\leq  \frac{1}{4}\int |H_t|^2 \dif x+ \|\nabla u\|_{L^3}^2\|H\|_{L^6}^2+ \|u\|_{L^\infty}^2\|\nabla H\|_{L^2}^2\nonumber\\
&\leq  \frac{1}{4}\int |H_t|^2 \dif x+ C\|\nabla u\|_{L^2}\|\nabla^2 u\|_{L^2}\|\nabla H\|_{L^2}^2,
\end{align}
which together with Gronwall's inequality, (\ref{3.1}) and (\ref{3.5}) implies
\begin{align}\label{3.20}
\sup_{0\leq t\leq T}\|\nabla{H}\|^2_{L^2}+\int_0^T\|H_t\|_{L^2}^2\dif t\leq~ &\|\nabla H_0\|_{L^2}^2\exp\left\{C\int_0^T \|\nabla u\|_{L^2}\|\nabla^2 u\|_{L^2}\dif t\right\}\nonumber\\
\leq~& \|\nabla H_0\|_{L^2}^2\exp\left\{C_3\tilde{\rho}^{\frac{1}{2}-\alpha+\frac{\gamma}{2}+\frac{\delta}{2}}\right\} \nonumber\\
\leq~&2\|\nabla H_0\|_{L^2}^2,
\end{align}
provided that $\alpha>\frac{1}{2}\gamma+\frac{1}{2}$ and $\tilde{\rho}\geq M_{22}\triangleq(\frac{C_3}{\ln 2})^{\frac{1}{2\alpha-\gamma-1-\delta}}$. Hence, taking $M_{2}=\max\{M_{21},M_{22}\}$, Lemma \ref{lem:3.4} is proved provided $\tilde{\rho}\geq M_{2}$.
\end{proof}
\end{lemma}

\begin{lemma}\label{lemma:2.5}
Assume that \eqref{3.09} holds, then we have 
\begin{align*}
\sup_{0\leq t\leq T}\left(\|\sqrt\rho u_t\|^2_{L^2}+\| H_t\|^2_{L^2}\right)+\int_0^T\left(\tilde{\rho}^\alpha\|\nabla u_t\|^2_{L^2}+\nu\|\nabla H_t\|^2_{L^2}\right)\dif t\leq \tilde{\rho}^{2\alpha-1+\delta},
\end{align*}
provided there exists a constant  $M_3$, such that $\tilde{\rho}\geq M_3$.
\end{lemma}
\begin{proof} 
Differentiating $(\ref{1.1})_2$ with respect to $t$ and taking $L^2$ inner product with $u_t$, one has
\begin{align}\label{3.21}
& \frac{1}{2}\frac{\dif}{\dif t}\int \rho\left|u_t\right|^2 \dif x+\int\left(2 \mu \rho^\alpha\left|\mathcal{D}\left(u_t\right)\right|^2+\lambda \rho^\alpha\left(\operatorname{div} u_t\right)^2\right)\dif x \nonumber\\
=~&\int \operatorname{div}(\rho u) \frac{u_t}{2} \cdot u_t\dif x+\int \operatorname{div}(\rho u) u \cdot \nabla u \cdot u_t\dif x-\int \rho u_t \cdot \nabla u \cdot u_t\dif x \nonumber\\
&+\int \operatorname{div}\left(\mu \alpha \rho^{\alpha-1} \rho_t \mathcal{D}(u)\right) \cdot u_t\dif x+\int \nabla\left(\lambda \alpha \rho^{\alpha-1} \rho_t \operatorname{div} u\right) \cdot u_t\dif x+\int P_t \operatorname{div} u_t \dif x \nonumber\\
&+\int\left( (\nabla \times H_t)\times H\right)\cdot u_t \dif x+\int\left( (\nabla \times H)\times H_t\right)\cdot u_t \dif x \nonumber\\
\triangleq~&\sum_{i=1}^{8} J_i,
\end{align}
where $J_i ~ (i=1,2,\cdots 8)$ can be estimated as:
\begin{align*}
J_1+J_3&=\int \operatorname{div}(\rho u) \frac{u_t}{2} \cdot u_t \dif x-\int \rho u_t \cdot \nabla u \cdot u_t \dif x\\
 &\leq C\tilde{\rho}\left\|\nabla u\right\|_{L^6}\left\|u_t\right\|_{L^2}\left\|u_t\right\|_{L^3}+C\left\|u\right\|_{L^\infty}\left\|\nabla \rho\right\|_{L^4}\left\|u_t\right\|_{L^2}\left\|u_t\right\|_{L^4}\\
  &\leq C\tilde{\rho}\left\|\nabla^2 u\right\|_{L^2}\left\|u_t\right\|_{L^2}^\frac{3}{2}\left\|\nabla u_t\right\|_{L^2}^\frac{1}{2}+C\left\|\nabla u\right\|_{L^2}^\frac{1}{2}\left\|\nabla^2u\right\|_{L^2}^\frac{1}{2}\left\|\nabla \rho\right\|_{L^4}\left\|u_t\right\|_{L^2}^\frac{5}{4}\left\|\nabla u_t\right\|_{L^2}^\frac{3}{4}\\
 &\leq \frac{\mu}{2^{\alpha+3}} \tilde{\rho}^\alpha\left\|\nabla u_t\right\|_{L^2}^2+\frac{C}{\tilde{\rho}^{\frac{\alpha}{3}-\frac{1}{3}}}\|\nabla^2 u\|_{L^2}^\frac{4}{3}\left\|\sqrt{\rho} u_t\right\|_{L^2}^2+\frac{C}{\tilde{\rho}^{\frac{3}{5}\alpha+1}}\left\|\nabla u\right\|_{L^2}^\frac{4}{5}\left\|\nabla^2u\right\|_{L^2}^\frac{4}{5}\left\|\nabla \rho\right\|_{L^4}^\frac{8}{5}\left\|\sqrt{\rho} u_t\right\|_{L^2}^2\\
  &\leq \frac{\mu}{2^{\alpha+3}} \tilde{\rho}^\alpha\left\|\nabla u_t\right\|_{L^2}^2+\frac{C}{\tilde{\rho}^{\frac{\alpha}{3}-\frac{1}{12}-\frac{\gamma}{4}-\frac{10}{3}\delta}}\left\|\sqrt{\rho} u_t\right\|_{L^2}^2+\frac{C}{\tilde{\rho}^{\frac{3}{5}\alpha-\frac{3}{20}-\frac{18}{5}\delta-\frac{9}{20}\gamma}}\left\|\sqrt{\rho} u_t\right\|_{L^2}^2,
\end{align*}
\begin{align*}
J_2&=\int \operatorname{div}(\rho u) u \cdot \nabla u \cdot u_t \dif x\\
&\leq C\tilde{\rho}\left\| u\right\|_{L^\infty}\left\|\nabla u\right\|_{L^4}^2\left\|u_t\right\|_{L^2}+C\left\|u\right\|_{L^\infty}^2\left\|\nabla \rho\right\|_{L^4}\left\|\nabla u\right\|_{L^4}\left\|u_t\right\|_{L^2}\\
&\leq C\tilde{\rho}\left\| \nabla u\right\|_{L^2}\left\| \nabla^2u\right\|_{L^2}^2\left\|u_t\right\|_{L^2}+C\left\| \nabla u\right\|_{L^2}^\frac{5}{4}\left\| \nabla^2u\right\|_{L^2}^\frac{7}{4}\left\|\nabla \rho\right\|_{L^4}\left\|u_t\right\|_{L^2}\\
&\leq C\tilde{\rho}^{\frac{1}{8}+\frac{3}{8}\gamma+5\delta}\left\| \nabla u\right\|_{L^2}\left\|\sqrt{\rho}u_t\right\|_{L^2}+C\tilde{\rho}^{\frac{23}{8}\delta+\frac{45}{64}\gamma-\frac{13}{64}}\left\| \nabla u\right\|_{L^2}\left\|\sqrt{\rho}u_t\right\|_{L^2},
\end{align*}

\begin{align*}
J_4+J_5
&=\int \operatorname{div}\left(\mu \alpha \rho^{\alpha-1} \rho_t \mathcal{D}(u)\right) \cdot u_t \dif x+\int \nabla\left(\lambda \alpha \rho^{\alpha-1} \rho_t \operatorname{div} u\right) \cdot u_t \dif x\\
&\leq C\tilde{\rho}^\alpha\left\|\nabla u\right\|_{L^4}^2\left\|\nabla u_t\right\|_{L^2}+C\tilde{\rho}^{\alpha-1}\left\|u\right\|_{L^\infty}\left\|\nabla \rho\right\|_{L^4}\left\|\nabla u\right\|_{L^4}\left\|\nabla u_t\right\|_{L^2}\\
&\leq \frac{\mu}{2^{\alpha+3}} \tilde{\rho}^\alpha\left\|\nabla u_t\right\|_{L^2}^2+C \tilde{\rho}^\alpha\|\nabla u\|_{L^2}\left\|\nabla^2 u\right\|_{L^2}^3+C\tilde{\rho}^{\alpha+\delta}\|\nabla u\|_{L^2}^\frac{3}{2}\left\|\nabla^2 u\right\|_{L^2}^\frac{5}{2}\\
&\leq \frac{\mu}{2^{\alpha+3}} \tilde{\rho}^\alpha\left\|\nabla u_t\right\|_{L^2}^2+C \tilde{\rho}^{\alpha+\frac{3}{8}(\gamma-1)+5\delta}\|\nabla u\|_{L^2}\left\|\nabla^2 u\right\|_{L^2}+C\tilde{\rho}^{\alpha+\frac{15}{32}\gamma-\frac{15}{32}+\frac{21}{4}\delta}\|\nabla u\|_{L^2}\left\|\nabla^2 u\right\|_{L^2},
\end{align*}

\begin{align*}
J_6&=\int P_t \operatorname{div} u_t \dif x\\
&\leq C\tilde{\rho}^\gamma\left\|\nabla u\right\|_{L^2}\left\|\nabla u_t\right\|_{L^2}+C\tilde{\rho}^{\gamma-1}\left\|u\right\|_{L^\infty}\left\|\nabla \rho\right\|_{L^2}\left\|\nabla u_t\right\|_{L^2}\\
&\leq \frac{\mu}{2^{\alpha+3}} \tilde{\rho}^\alpha\left\|\nabla u_t\right\|_{L^2}^2+C \tilde{\rho}^{2 \gamma-\alpha}\|\nabla u\|_{L^2}^2+C \tilde{\rho}^{2\gamma-\alpha+\delta}\|\nabla u\|_{L^2}\|\nabla^2 u\|_{L^2},
\end{align*}

\begin{align*}
J_7+J_8=
&\int\left( (\nabla \times H_t)\times H\right)\cdot u_t \dif x+\int\left( (\nabla \times H)\times H_t\right)\cdot u_t \dif x\\
&\leq C\|\nabla H_t\|_{L^2}\|H\|_{L^3}\|u_t\|_{L^6}+ C\| H_t\|_{L^3}\|\nabla H\|_{L^2}\|u_t\|_{L^6}\\
&\leq C\|\nabla H_t\|_{L^2}\|H\|_{L^2}^{\frac{1}{2}}\|\nabla H\|_{L^2}^{\frac{1}{2}}\|\nabla u_t\|_{L^2}+C\|H_t\|_{L^2}^{\frac{1}{2}}\|\nabla H_t\|_{L^2}^{\frac{1}{2}}\|\nabla H\|_{L^2}\|\nabla u_t\|_{L^2}\\
&\leq \frac{\mu}{2^{\alpha+3}} \tilde{\rho}^\alpha\left\|\nabla u_t\right\|_{L^2}^2+C\tilde{\rho}^{-\alpha}\|\nabla H_t\|_{L^2}^2+C\tilde{\rho}^{-\alpha}\|H_t\|_{L^2}\|\nabla H_t\|_{L^2}.
\end{align*}
Inserting the above estimates into \eqref{3.21}, and integrating the resulting equality with respect to $t$, we have
\begin{align*}
&\sup_{0 \leq t \leq T} \int \rho\left|u_t\right|^2 d x+\tilde{\rho}^\alpha \int_0^T \int\left|\nabla u_t\right|^2 d x d t \nonumber\\
\leq~&\|\sqrt{\rho_0} u_{t 0} \|_{L^2}^2 +C\left(\frac{1}{\tilde{\rho}^{\frac{\alpha}{3}-\frac{1}{12}-\frac{\gamma}{4}-\frac{10}{3}\delta}}+\frac{1}{\tilde{\rho}^{\frac{3}{5}\alpha-\frac{3}{20}-\frac{18}{5}\delta-\frac{9}{20}\gamma}}\right)\int_0^T\left\|\sqrt{\rho} u_t\right\|_{L^2}^2\dif t \nonumber\\
&+C\left(\tilde{\rho}^{\frac{1}{8}+\frac{3}{8}\gamma+5\delta}+\tilde{\rho}^{\frac{23}{8}\delta+\frac{45}{64}\gamma-\frac{13}{64}}\right)\int_0^T\left\| \nabla u\right\|_{L^2}\left\|\sqrt{\rho}u_t\right\|_{L^2}\dif t\nonumber\\
&+C\tilde{\rho}^{\alpha+\frac{15}{32}\gamma-\frac{15}{32}+\frac{21}{4}\delta}\int_0^T\|\nabla u\|_{L^2}\left\|\nabla^2 u\right\|_{L^2}\dif t+C\tilde{\rho}^{2\gamma-\alpha}\int_0^T\|\nabla u\|_{L^2}^2\dif t\nonumber\\
&+C\tilde{\rho}^{-\alpha}\int_0^T\|\nabla H_t\|_{L^2}^2\dif t+C\tilde{\rho}^{-\alpha}\int_0^T\|H_t\|_{L^2}\|\nabla H_t\|_{L^2}\dif t\nonumber\\
\leq~&\|\sqrt{\rho_0} u_{t 0} \|_{L^2}^2 +C\left(\tilde{\rho}^{\frac{2}{3}\alpha-\frac{2}{3}+\frac{16}{3}\delta+\gamma}+\tilde{\rho}^{\frac{2}{5}\alpha-\frac{3}{5}+\frac{28}{5}\delta+\frac{6}{5}\gamma}\right)+C\left(\tilde{\rho}^{\frac{1}{4}+6\delta+\frac{3}{4}\gamma}+\tilde{\rho}^{4\delta+\frac{69}{64}\gamma-\frac{5}{64}}\right)\nonumber\\
&+C\tilde{\rho}^{\frac{31}{32}\gamma+\frac{1}{32}+6\delta}+C\tilde{\rho}^{2 \gamma-2\alpha+1}+C\tilde{\rho}^{\alpha-1+\delta}+C\tilde{\rho}^{\frac{\delta}{2}-\frac{1}{2}}\nonumber\\
\leq~& C_4\tilde{\rho}^{2\alpha-1} \quad \nonumber\\
\leq~&\tilde{\rho}^{2\alpha-1+\delta},
\end{align*}
provided that $\alpha>\frac{3}{4}\gamma+\frac{1}{4}$ and $\tilde{\rho}\geq M_{31}\triangleq C_4^{\frac{1}{\delta}}$, where in the last inequality, the initial value and it's $L^2$-norm of the second-order derivative of the velocity can be given by 
\begin{align*}
u_{t 0} \triangleq \rho_0^{-1}\left(\operatorname{div}\left(2 \mu \rho_0^\alpha \mathcal{D}\left(u_0\right)\right)+\nabla\left(\lambda \rho_0^{\alpha}\operatorname{div} u_0 \right)-\rho_0 u_0 \cdot \nabla u_0-\nabla P\left(\rho_0\right)\right),
\end{align*}
and
\begin{align*}
\left\|\sqrt{\rho_0} u_{t 0}\right\|_{L^2}^2 \leq C\left(\tilde{\rho}^{2 \alpha-3}+\tilde{\rho}^{2 \alpha-1}+\tilde{\rho}^{2 \gamma-3}+\tilde{\rho}\right)\leq C\tilde{\rho}^{2 \alpha-1}.
\end{align*}

Next, differentiating $(\ref{1.1})_3$ with respect to $t$, one has
\begin{align*}
H_{tt}-\nabla\times(u_t\times H)-\nabla\times(u\times H_t)
=-\nabla\times(\nu\nabla\times H_t)
\end{align*}

Taking $L^2$ inner product with $H_t$, we obtain
\begin{align*}
&\frac{1}{2}\frac{\dif}{\dif t}\int|H_t|^{2}\dif x+\int\nu|\nabla H_t|^{2}\dif x\\
\leq~&\int |\nabla u_t|  |H| |H_t| \dif x+\int  |u_t| |\nabla H| |H_t| \dif x+\int |\nabla u|  |H_t|^2 \dif x+\int  |u| |\nabla H_t| |H_t| \dif x\\
\leq~&\|\nabla u_t\|_{L^2}\|H\|_{L^3}\|H_t\|_{L^6}+\|  u_t\|_{L^2}\|\nabla H\|_{L^3}\|H_t\|_{L^6}+\|\nabla u \|_{L^2}\|H_t\|_{L^3}\|H_t\|_{L^6}+\|  u \|_{L^6}\|\nabla H_t\|_{L^2}\|H_t\|_{L^3}\\
\leq~&C\left(\|\nabla u_t\|_{L^2}\|H\|_{L^2}^{\frac{1}{2}}\|\nabla H\|_{L^2}^{\frac{1}{2}}\|\nabla H_t\|_{L^2}
+\|  u_t\|_{L^2}\|\nabla H\|_{L^2}^{\frac{1}{2}}\|\nabla^2 H\|_{L^2}^{\frac{1}{2}}\|\nabla H_t\|_{L^2}+\| \nabla u \|_{L^2}\|H_t\|_{L^2}^{\frac{1}{2}}\|\nabla H_t\|_{L^2}^{\frac{3}{2}}\right)\\
\leq~&\frac{\nu}{2}\|\nabla H_t\|_{L^2}^{2}
+C\|\nabla u_t\|_{L^2}^{2}\|  H \|_{L^2} \|\nabla H \|_{L^2}+C\|  u_t\|_{L^2}^2\|\nabla H\|_{L^2}\|\nabla^2 H\|_{L^2}+C\| \nabla u \|_{L^2}^4\|H_t\|_{L^2}^2\\
\leq~&\frac{\nu}{2}\|\nabla H_t\|_{L^2}^{2}
+C\|\nabla u_t\|_{L^2}^{2}+C\|  u_t\|_{L^2}^2\left(\|H_t\|_{L^2}+\|\nabla u\|_{L^2}^2\|\nabla H\|_{L^2}+\|\nabla H\|_{L^2}\|\nabla u\|_{L^2}^\frac{1}{2}\|\nabla^2 u\|_{L^2}^\frac{1}{2}\right)\\
&+C\tilde{\rho}^{\frac{3}{2}\gamma-\frac{3}{2}+4\delta}\|H_t\|_{L^2}^2\\
\leq~&\frac{\nu}{2}\|\nabla H_t\|_{L^2}^{2}
+C\|\nabla u_t\|_{L^2}^{2}+C\tilde{\rho}^{\alpha-\frac{3}{2}+\frac{\delta}{2}}\| \sqrt{\rho}u_t\|_{L^2}\|H_t\|_{L^2}+C\left(\tilde{\rho}^{\frac{3}{4}\gamma-\frac{7}{4}+2\delta}+\tilde{\rho}^{\frac{1}{4}\alpha+\frac{3}{8}\gamma+2\delta-\frac{11}{8}}\right)\| \sqrt{\rho}u_t\|_{L^2}^2\\
&+C\tilde{\rho}^{\frac{3}{2}\gamma-\frac{3}{2}+4\delta}\|H_t\|_{L^2}^2.
\end{align*}

Integrating the above estimate over $(0,T)$ and noting that
$$H_{t0}=\nabla\times(u_0\times H_0)-\nu\nabla\times(\nabla\times H_0),$$
we have
\begin{align*}
&\|H_t\|_{L^2}^{2}+\int_{0}^{T}\|\nabla H_t\|_{L^2}^{2}\dif t\\
\leq~& \|H_{t0}\|_{L^2}^{2}+
C\int_{0}^{T}\|\nabla u_t\|_{L^2}^{2}\dif t+C\tilde{\rho}^{\alpha-\frac{3}{2}+\frac{\delta}{2}}\int_0^T\| \sqrt{\rho}u_t\|_{L^2}\|H_t\|_{L^2}\dif t\\
&+C\left(\tilde{\rho}^{\frac{3}{4}\gamma-\frac{7}{4}+2\delta}+\tilde{\rho}^{\frac{1}{4}\alpha+\frac{3}{8}\gamma+2\delta-\frac{11}{8}}\right)\int_0^T\| \sqrt{\rho}u_t\|_{L^2}^2\dif t+C\tilde{\rho}^{\frac{3}{2}\gamma-\frac{3}{2}+4\delta}\int_{0}^{T}\|H_t\|_{L^2}^2\dif t\\
\leq~& \|H_{t0}\|_{L^2}^{2}+
C\left(\tilde{\rho}^{\alpha-1+\delta}
+\tilde{\rho}^{\frac{3}{2}\alpha+\frac{3}{2}\delta+\frac{3}{8}\gamma-\frac{15}{8}}+\tilde{\rho}^{\alpha+\frac{3}{2}\gamma-\frac{5}{2}+4\delta}+\tilde{\rho}^{\frac{5}{4}\alpha+\frac{9}{8}\gamma+4\delta-\frac{17}{8}}+\tilde{\rho}^{\frac{3}{2}\gamma-\frac{3}{2}+4\delta}\right)\\
\leq~& C_5\tilde{\rho}^{2\alpha-1} \\
\leq~&\tilde{\rho}^{2\alpha-1+\delta},
\end{align*}
provided that $\alpha>\frac{3}{2}\gamma-\frac{3}{2}$ and $\tilde{\rho}\geq M_{32}\triangleq  C_5^{\frac{1}{\delta}}$. Taking $M_3=\max\{M_{31}, M_{32}\}$, we complete the proof.
\end{proof}

\begin{lemma}\label{D:rho}
Assume that \eqref{3.09} holds, then we have 
\begin{equation}\label{2.24}
\sup_{0\leq t\leq T}\left(\|\nabla \rho\|_{L^2}^2+\|\nabla \rho\|_{L^4}^2\right)+\tilde{\rho}^{\gamma-\alpha}\int_0^T\left(\|\nabla \rho\|_{L^2}^2+\|\nabla \rho\|_{L^4}^2\right)\dif t \leq 2\tilde{\rho}^{2+\delta},
\end{equation}
\begin{equation}
\sup_{0\leq t\leq T}\|\rho_0-\tilde{\rho}\|_{L^\infty}\leq\frac{\tilde{\rho}}{3},
\end{equation}
provided there exists a constant $M_4$, such that $\tilde{\rho}\geq M_4$.
\end{lemma}
\begin{proof}
Applying $\nabla$ operator on $(\ref{1.1})_1$ and multiplying $\nabla \rho^{r-1}$ on both sides of the resulting equation, we obtain
\begin{align}\label{lr}
\frac{\dif}{\dif t}\|\nabla\rho\|_{L^r}^{r}+\tilde{\rho}^{\gamma-\alpha}\|\nabla\rho\|_{L^r}^{r}\leq C\|\nabla u\|_{L^{\infty}}\|\nabla \rho\|_{L^r}^{r}+C\tilde{\rho}\|\nabla F\|_{L^r}\|\nabla\rho\|_{L^r}^{r-1}+C\tilde{\rho}\|H\cdot\nabla H\|_{L^r}\|\nabla\rho\|_{L^r}^{r-1},
\end{align}
where $r\geq 2$, and $\displaystyle F=(2\mu+\lambda)\text{div}u- (\mathcal{P}(\rho)-\mathcal{P}(\tilde{\rho}))-\frac{1}{2}|H|^2$ is the effective viscous flux.

When $r=2$, (\ref{lr}) reduces to 
\begin{align}\label{l2}
&\frac{\dif}{\dif t}\|\nabla\rho\|_{L^2}^2+\tilde{\rho}^{\gamma-\alpha}\|\nabla\rho\|_{L^2}^2\nonumber\\
\leq~& C\left(\|\nabla u\|_{L^{\infty}}\|\nabla \rho\|_{L^2}^2+\tilde{\rho}\|\nabla F\|_{L^2}\|\nabla\rho\|_{L^2}+\tilde{\rho}\|H\cdot\nabla H\|_{L^2}\|\nabla\rho\|_{L^2}\right)\nonumber\\
\leq~& \frac{1}{2}\tilde{\rho}^{\gamma-\alpha}\|\nabla\rho\|_{L^2}^2+C\left(\tilde{\rho}^{2+\alpha+\delta-\gamma}\|\nabla u\|_{L^{\infty}}^2+\tilde{\rho}^{\alpha-\gamma+2}\|\nabla F\|_{L^2}^2+\tilde{\rho}\|H\|_{L^\infty}\|\nabla H\|_{L^2}\|\nabla\rho\|_{L^2}\right)\nonumber\\
\leq~& \frac{1}{2}\tilde{\rho}^{\gamma-\alpha}\|\nabla\rho\|_{L^2}^2+C\left(\tilde{\rho}^{2+\alpha+\delta-\gamma}\|\nabla u\|_{L^2}^\frac{2}{7}\|\nabla^2 u\|_{L^{4}}^\frac{12}{7}+\tilde{\rho}^{\alpha-\gamma+2}\|\nabla F\|_{L^2}^2+\tilde{\rho}\|\nabla H\|_{L^2}^\frac{3}{2}\|\nabla^2 H\|_{L^{2}}^\frac{1}{2}\|\nabla\rho\|_{L^2}\right).
\end{align}

In order to estimate $\|\nabla F\|_{L^r}$, we note that 
$$-\Delta F=\text{div}g,$$
where
$$g:=\rho^{-\alpha}\left(-\rho u_t -\rho u\cdot \nabla u+ 2\nabla \rho^\alpha\cdot \mathcal {D}(u)+\nabla \rho^\alpha\cdot \text{div} u\mathbb{I}_3+H\cdot\nabla H-\frac{1}{2}\nabla |H|^2\right).$$
Then, we have
\begin{align}\label{f2}
&\|\nabla F\|_{L^2}\leq\|g\|_{L^2}\nonumber\\
\leq~& \frac{C}{\tilde{\rho}^{\alpha-\frac{1}{2}}}\|\sqrt{\rho}u_{t}\|_{L^{2}}
+\frac{C}{\tilde{\rho}^{\alpha-1}}\|u \|_{L^{6}}\|\nabla u \|_{L^{3}}
+\frac{C}{\tilde{\rho}}\|\nabla \rho \|_{L^{4}}\|\nabla u \|_{L^{4}}
+\frac{C}{\tilde{\rho}^{\alpha }}\|H\|_{L^{6}}\|\nabla H \|_{L^{3}}\nonumber\\
\leq~& \frac{C}{\tilde{\rho}^{\alpha-\frac{1}{2}}}\|\sqrt{\rho}u_{t}\|_{L^{2}}
+\frac{C}{\tilde{\rho}^{\alpha-1}}\|\nabla u \|_{L^{2}}^\frac{3}{2}\|\nabla^2 u \|_{L^{2}}^\frac{1}{2}
+C\tilde{\rho}^{\frac{\delta}{2}}\|\nabla u \|_{L^{2}}^\frac{1}{4}\|\nabla^2 u \|_{L^{2}}^\frac{3}{4}
+\frac{C}{\tilde{\rho}^{\alpha }}\|\nabla H\|_{L^{2}}^\frac{3}{2}\|\nabla^2 H \|_{L^{2}}^\frac{1}{2}.
\end{align}
It remains to estimate $\|\nabla^2 u\|_{L^{4}}$ on the right-hand side of (\ref{l2}). From $(\ref{1.1})_2$, one has
\begin{align}\label{f2-11}
\|\nabla^{2} u \|_{L^{4}}\leq~&\frac{C}{\tilde{\rho}^{\alpha}}
\left(\tilde{\rho}\|u_t\|_{L^4}+\tilde{\rho}\|u\cdot\nabla u\|_{L^4}+\tilde{\rho}^{\alpha-1}\|\nabla\rho\nabla u\|_{L^4}
+ \|H\cdot\nabla H\|_{L^4}\right)+\frac{C}{\tilde{\rho}^{\alpha+1-\gamma}}\|\nabla\rho\|_{L^4}\nonumber\\
\leq~&\frac{C}{\tilde{\rho}^{\alpha-1}}\|u_t\|_{L^4}
+\frac{C}{\tilde{\rho}^{\alpha+1-\gamma}}\|\nabla\rho\|_{L^4}
+\frac{C}{\tilde{\rho}^{\alpha-1}}\|u \|_{L^\infty}\|\nabla u \|_{L^4}+\frac{C}{\tilde{\rho} }\|\nabla\rho \|_{L^4}\|\nabla u\|_{L^\infty}\nonumber\\
&+\frac{C}{\tilde{\rho}^{\alpha }}\|H \|_{L^\infty} \|\nabla H \|_{L^4}\nonumber\\
\leq~&\frac{C}{\tilde{\rho}^{\alpha-1}}\|u_t\|_{L^4}
+\frac{C}{\tilde{\rho}^{\alpha+1-\gamma}}\|\nabla\rho\|_{L^4}
+\frac{C}{\tilde{\rho}^{\alpha-1}}
\|\nabla u \|_{L^{2}}^{\frac{3}{4}}
\|\nabla^{2} u \|_{L^{2}}^{\frac{5}{4}}\nonumber\\
&+\frac{C}{\tilde{\rho} }\|\nabla \rho \|_{L^4}
\|\nabla u \|_{L^{2}}^{\frac{1}{7}}
\|\nabla^{2} u \|_{L^{4}}^{\frac{6}{7}}
+\frac{C}{\tilde{\rho}^{\alpha }} \|\nabla H \|_{L^{2}}^{\frac{3}{4}}
\|\nabla^{2} H \|_{L^{2}}^{\frac{5}{4}}\nonumber\\
\leq~&\frac{1}{2}\|\nabla^{2} u \|_{L^{4}}+\frac{C}{\tilde{\rho}^{\alpha-1}}\|u_t\|_{L^2}^{\frac{1}{4}}\|\nabla u_t\|_{L^2}^{\frac{3}{4}}
+\frac{C}{\tilde{\rho}^{\alpha+1-\gamma}}\|\nabla\rho\|_{L^4}
+\frac{C}{\tilde{\rho}^{\alpha-1}}
\|\nabla u \|_{L^{2}}^{\frac{3}{4}}
\|\nabla^{2} u \|_{L^{2}}^{\frac{5}{4}}\nonumber\\
&+\frac{C}{\tilde{\rho}^7}\|\nabla \rho \|_{L^4}^7
\|\nabla u \|_{L^{2}}
+\frac{C}{\tilde{\rho}^{\alpha }} \|\nabla H \|_{L^{2}}^{\frac{3}{4}}
\|\nabla^{2} H \|_{L^{2}}^{\frac{5}{4}}.
\end{align}


Inserting the above inequalities (\ref{f2}) and (\ref{f2-11}) into (\ref{l2}), and integrating the resulting inequality over $(0,T)$, we have
\begin{align}\label{l2-1}
&\|\nabla\rho\|_{L^2}^2+\tilde{\rho}^{\gamma-\alpha}\int_{0}^{T}\|\nabla\rho\|_{L^2}^2\dif t\nonumber\\
\leq~& \|\nabla\rho_0\|_{L^2}^2+C\tilde{\rho}^{2+\alpha+\delta-\gamma}\int_{0}^{T}\|\nabla u\|_{L^2}^\frac{2}{7}\|\nabla^2 u\|_{L^{4}}^\frac{12}{7}\dif t+C\tilde{\rho}^{\alpha-\gamma+2}\int_{0}^{T}\|\nabla F\|_{L^2}^2\dif t\nonumber\\
&+C\tilde{\rho}^2\int_0^T\|\nabla H\|_{L^2}\|\nabla^2 H\|_{L^2}\dif t+\int_{0}^{T}\|\nabla H\|_{L^2}^2\|\nabla\rho\|_{L^2}^2\dif t\nonumber\\
\leq~& \|\nabla\rho_0\|_{L^2}^2+C\tilde{\rho}^{2+\alpha+\delta-\gamma}\left(\int_{0}^{T}\|\nabla u\|_{L^2}^2\dif t\right)^\frac{1}{7}\left(\int_{0}^{T}\|\nabla^2 u\|_{L^{4}}^2\dif t\right)^\frac{6}{7}+C\tilde{\rho}^{\alpha-\gamma+2}\int_{0}^{T}\|\nabla F\|_{L^2}^2\dif t\nonumber\\
&+C\tilde{\rho}^2+\int_{0}^{T}\|\nabla H\|_{L^2}^2\|\nabla\rho\|_{L^2}^2\dif t\nonumber\\
\leq~& \|\nabla\rho_0\|_{L^2}^2+C\tilde{\rho}^{\frac{15}{7}-\frac{\gamma}{7}+\frac{13}{7}\delta}+C\tilde{\rho}^{\frac{9}{4}-\frac{\gamma}{4}+2\delta}+C\tilde{\rho}^2+\int_{0}^{T}\|\nabla H\|_{L^2}^2\|\nabla\rho\|_{L^2}^2\dif t,
\end{align}
where we have used 
\begin{align}\label{p0}
\int_{0}^{T}\|\nabla^2 u \|_{L^4}^{2}\dif t
\leq~&
\frac{C}{\tilde{\rho}^{2\alpha-2}}\int_{0}^{T}\|u_t\|_{L^2}^{\frac{1}{2}}\|\nabla u_t\|_{L^2}^{\frac{3}{2}}\dif t
+\frac{C}{\tilde{\rho}^{2\alpha+2-2\gamma}}\int_{0}^{T}\|\nabla\rho\|_{L^4}^2\dif t
+\frac{C}{\tilde{\rho}^{2\alpha-2}}
\int_{0}^{T}\|\nabla u \|_{L^{2}}^{\frac{3}{2}}
\|\nabla^{2} u \|_{L^{2}}^{\frac{5}{2}}\dif t\nonumber\\
&+\frac{C}{\tilde{\rho}^{14}}\int_{0}^{T}\|\nabla \rho \|_{L^4}^{14}
\|\nabla u \|_{L^{2}}^2\dif t
+\frac{C}{\tilde{\rho}^{2\alpha }} \int_{0}^{T}\|\nabla H \|_{L^{2}}^{\frac{3}{2}}
\|\nabla^{2} H \|_{L^{2}}^{\frac{5}{2}}\dif t\nonumber\\
\leq~&
\frac{C}{\tilde{\rho}^{2\alpha-\frac{7}{4}}}\int_{0}^{T}
\|\sqrt{\rho}u_t\|_{L^2}^{\frac{1}{2}}\|\nabla u_t\|_{L^2}^{\frac{3}{2}}\dif t
+\frac{C}{\tilde{\rho}^{\alpha-\gamma-\delta}}
+\frac{C}{\tilde{\rho}^{2\alpha-\frac{5}{4}-\frac{3}{4}\gamma-3\delta}}
\int_{0}^{T}
\|\nabla^{2} u \|_{L^{2}}^2\dif t\nonumber\\
&+\frac{C}{\tilde{\rho}^{\alpha-1-7\delta}}
+\frac{C}{\tilde{\rho}^{\frac{3}{2}\alpha+\frac{1}{4}-\frac{1}{4}\delta}}\nonumber\\
\leq~&
\frac{C}{\tilde{\rho}^{\alpha-\frac{3}{16}\gamma-\frac{13}{16}-\frac{5}{4}\delta}}
+\frac{C}{\tilde{\rho}^{\alpha-\gamma-\delta}}
+\frac{C}{\tilde{\rho}^{3\alpha-\frac{5}{4}-\frac{7}{4}\gamma-4\delta}}
+\frac{C}{\tilde{\rho}^{\alpha-1-7\delta}}
+\frac{C}{\tilde{\rho}^{\frac{3}{2}\alpha+\frac{1}{4}-\frac{\delta}{4}}}\nonumber\\
\leq~&
\frac{C}{\tilde{\rho}^{\alpha-\gamma-\delta}},
\end{align}
provided that $\alpha>\frac{3}{8}\gamma+\frac{5}{8}$, and 
\begin{align}\label{f2-1}
\int_{0}^{T}\|\nabla F\|_{L^2}^2\dif t
\leq~& \frac{C}{\tilde{\rho}^{2\alpha-1}}\int_{0}^{T}\|\sqrt{\rho}u_{t}\|_{L^{2}}^2\dif t
+\frac{C}{\tilde{\rho}^{2\alpha-2}}\int_{0}^{T}\|\nabla u \|_{L^{2}}^3\|\nabla^2 u \|_{L^{2}}\dif t\nonumber\\
&+C\tilde{\rho}^{\delta}\int_{0}^{T}\|\nabla u \|_{L^{2}}^\frac{1}{2}\|\nabla^2 u \|_{L^{2}}^\frac{3}{2}\dif t
+\frac{C}{\tilde{\rho}^{2\alpha }}\int_{0}^{T}\|\nabla H\|_{L^{2}}^3\|\nabla^2 H \|_{L^{2}}\dif t\nonumber\\
\leq~& \frac{C}{\tilde{\rho}^{\alpha-\frac{3}{4}\gamma-\frac{1}{4}-2\delta}}+\frac{C}{\tilde{\rho}^{3\alpha-\frac{5}{4}\gamma-\frac{7}{4}-\frac{5}{2}\delta}}+\frac{C}{\tilde{\rho}^{\alpha-\frac{3}{4}\gamma-\frac{1}{4}-\frac{7}{4}\delta}}+\frac{C}{\tilde{\rho}^{2\alpha}}\nonumber\\
\leq~& \frac{C}{\tilde{\rho}^{\alpha-\frac{3}{4}\gamma-\frac{1}{4}-2\delta}},
\end{align}
provided that $\alpha>\frac{1}{4}\gamma+\frac{3}{4}$. Then, by using Gronwall's inequality, from (\ref{l2-1}) and (\ref{3.1-1}), we have
\begin{align}\label{l2-2}
&\|\nabla\rho\|_{L^2}^2+\tilde{\rho}^{\gamma-\alpha}\int_{0}^{T}\|\nabla\rho\|_{L^2}^2\dif t\nonumber\\
\leq~& \left(\|\nabla\rho_0\|_{L^2}^2+C\tilde{\rho}^{\frac{15}{7}-\frac{1}{7}\gamma+\frac{13}{7}\delta}+C\tilde{\rho}^{\frac{9}{4}-\frac{1}{4}\gamma+\frac{7}{4}\delta}+C\tilde{\rho}^2\right)\exp\left\{C_6\int_{0}^{T}\|\nabla H\|_{L^2}^2\dif t\right\}\leq\tilde{\rho}^{2+\delta},
\end{align}
provided that $\tilde{\rho}\geq M_{41}\triangleq(2C_6\exp\{\frac{2C_6}{\nu}\|\nabla H_0\|_{L^2}^2\})^{\frac{1}{\delta}}$.

When $r=4$, (\ref{lr}) reduces to 
\begin{align}\label{l4}
&\frac{\dif}{\dif t}\|\nabla\rho\|_{L^4}^2+\tilde{\rho}^{\gamma-\alpha}\|\nabla\rho\|_{L^4}^2\nonumber\\
\leq~& C\|\nabla u\|_{L^{\infty}}\|\nabla \rho\|_{L^4}^2+C\tilde{\rho}\|\nabla F\|_{L^4}\|\nabla\rho\|_{L^4}+C\tilde{\rho}\|H\cdot\nabla H\|_{L^4}\|\nabla\rho\|_{L^4}\nonumber\\
\leq~& \frac{1}{2}\tilde{\rho}^{\gamma-\alpha}\|\nabla\rho\|_{L^4}^2+C\tilde{\rho}^{\alpha-\gamma}\|\nabla\rho\|_{L^4}^2\|\nabla u\|_{L^{\infty}}^2+C\tilde{\rho}^{\alpha-\gamma+2}\|\nabla F\|_{L^4}^2+C\tilde{\rho}\|H\|_{L^\infty}\|\nabla H\|_{L^4}\|\nabla\rho\|_{L^4}\nonumber\\
\leq~& \frac{1}{2}\tilde{\rho}^{\gamma-\alpha}\|\nabla\rho\|_{L^4}^2+C\tilde{\rho}^{\alpha-\gamma+2+\delta}\|\nabla u\|_{L^2}^\frac{2}{7}\|\nabla^2 u\|_{L^{4}}^\frac{12}{7}+C\tilde{\rho}^{\alpha-\gamma+2}\|\nabla F\|_{L^4}^2+C\tilde{\rho}\|\nabla H\|_{L^2}^\frac{3}{4}\|\nabla^2 H\|_{L^{2}}^\frac{5}{4}\|\nabla\rho\|_{L^4},
\end{align}
where $\|\nabla F\|_{L^4}$ can be estimated as 
\begin{align}\label{f4}
&\|\nabla F\|_{L^4}\leq \|g\|_{L^4}\nonumber\\
\leq~&\frac{C}{\tilde{\rho}^{\alpha-1}}\|u_{t}\|_{L^{4}}
+\frac{C}{\tilde{\rho}^{\alpha-1}}\|u \|_{L^{\infty}}\|\nabla u \|_{L^{4}}
+\frac{C}{\tilde{\rho}}\|\nabla \rho \|_{L^{4}}\|\nabla u \|_{L^{\infty}}
+\frac{C}{\tilde{\rho}^{\alpha }}\|H\|_{L^{\infty}}\|\nabla H \|_{L^{4}}\nonumber\\
\leq~&\frac{C}{\tilde{\rho}^{\alpha-1}}\|u_{t}\|_{L^{2}}^\frac{1}{4}\|\nabla u_{t}\|_{L^{2}}^\frac{3}{4}
+\frac{C}{\tilde{\rho}^{\alpha-1}}\|\nabla u\|_{L^2}^\frac{3}{4}\|\nabla^2 u\|_{L^{2}}^\frac{5}{4}
+C\tilde{\rho}^{\frac{\delta}{2}}\|\nabla u \|_{L^{\infty}}
+\frac{C}{\tilde{\rho}^{\alpha}}\|\nabla H\|_{L^2}^\frac{3}{4}\|\nabla^2 H\|_{L^{2}}^\frac{5}{4}\nonumber\\
\leq~&\frac{C}{\tilde{\rho}^{\alpha-\frac{7}{8}}}\|\sqrt{\rho}u_{t}\|_{L^{2}}^\frac{1}{4}\|\nabla u_{t}\|_{L^{2}}^\frac{3}{4}
+\frac{C}{\tilde{\rho}^{\alpha-1}}\|\nabla u\|_{L^2}^\frac{3}{4}\|\nabla^2 u\|_{L^{2}}^\frac{5}{4}
+C\tilde{\rho}^{\frac{\delta}{2}}\|\nabla u\|_{L^2}^\frac{1}{7}\|\nabla^2 u\|_{L^{4}}^\frac{6}{7}+\frac{C}{\tilde{\rho}^{\alpha}}\|\nabla H\|_{L^2}^\frac{3}{4}\|\nabla^2 H\|_{L^{2}}^\frac{5}{4}.
\end{align}

Then, integrating (\ref{l4}) over $(0,T)$, we have
\begin{align}\label{l4-1}
&\|\nabla\rho\|_{L^4}^2+\tilde{\rho}^{\gamma-\alpha}\int_0^T\|\nabla\rho\|_{L^4}^2\dif t\nonumber\\
\leq~& \|\nabla\rho_0\|_{L^4}^2+C\tilde{\rho}^{\alpha-\gamma+2+\delta}\int_0^T\|\nabla u\|_{L^2}^\frac{2}{7}\|\nabla^2 u\|_{L^{4}}^\frac{12}{7}\dif t+C\tilde{\rho}\int_0^T\|\nabla H\|_{L^2}^\frac{3}{4}\|\nabla^2 H\|_{L^{2}}^\frac{5}{4}\|\nabla\rho\|_{L^4}\dif t\nonumber\\
&+\tilde{\rho}^{\alpha-\gamma+2}\int_0^T\bigg(\frac{C}{\tilde{\rho}^{2\alpha-\frac{7}{4}}}\|\sqrt{\rho}u_{t}\|_{L^{2}}^\frac{1}{2}\|\nabla u_{t}\|_{L^{2}}^\frac{3}{2}
+\frac{C}{\tilde{\rho}^{2\alpha-2}}\|\nabla u\|_{L^2}^\frac{3}{2}\|\nabla^2 u\|_{L^{2}}^\frac{5}{2}
+C\tilde{\rho}^{\delta}\|\nabla u\|_{L^2}^\frac{2}{7}\|\nabla^2 u\|_{L^{4}}^\frac{12}{7}\nonumber\\
&+\frac{C}{\tilde{\rho}^{2\alpha}}\|\nabla H\|_{L^2}^\frac{3}{2}\|\nabla^2 H\|_{L^{2}}^\frac{5}{2}\bigg)\dif t\nonumber\\
\leq~& \|\nabla\rho_0\|_{L^4}^2+C\tilde{\rho}^{\alpha-\gamma+2+\delta}\left(\int_0^T\|\nabla u\|_{L^2}^2\dif t\right)^\frac{1}{7}\left(\int_0^T\|\nabla^2 u\|_{L^{4}}^2\dif t\right)^\frac{6}{7}+C\tilde{\rho}^2\int_0^T\|\nabla H\|_{L^2}\|\nabla^2 H\|_{L^{2}}\dif t\nonumber\\
&+C\tilde{\rho}^{\frac{15}{4}-\alpha-\gamma}\left(\int_0^T\|\sqrt{\rho}u_{t}\|_{L^{2}}^2\dif t\right)^\frac{1}{4}\left(\int_0^T\|\nabla u_{t}\|_{L^{2}}^2\dif t\right)^\frac{3}{4}+C\tilde{\rho}^{\frac{13}{4}-\alpha-\frac{1}{4}\gamma+3\delta}\int_0^T\|\nabla^2 u\|_{L^2}^2\dif t\nonumber\\
&+C\tilde{\rho}^{\frac{7}{4}-\frac{\alpha}{2}-\gamma+\frac{\delta}{4}}\int_0^T\|\nabla^2 H\|_{L^2}^2\dif t+\int_0^T\|\nabla H\|_{L^2}^\frac{1}{2}\|\nabla^2 H\|_{L^{2}}^\frac{3}{2}\|\nabla\rho\|_{L^4}^2\dif t\nonumber\\
\leq~& \|\nabla\rho_0\|_{L^4}^2+C\tilde{\rho}^{\frac{15}{7}-\frac{1}{7}\gamma+\frac{13}{7}\delta}+C\tilde{\rho}^2+C\tilde{\rho}^{\frac{45}{16}-\frac{13}{16}\gamma+\frac{5}{4}\delta}+C\tilde{\rho}^{\frac{13}{4}-2\alpha+\frac{3}{4}\gamma+4\delta}+C\tilde{\rho}^{\frac{7}{4}-\frac{\alpha}{2}-\gamma+\frac{\delta}{4}}\nonumber\\
&+\int_0^T\|\nabla H\|_{L^2}^\frac{1}{2}\|\nabla^2 H\|_{L^{2}}^\frac{3}{2}\|\nabla\rho\|_{L^4}^2\dif t,
\end{align}
which together with Gronwall's inequality and (\ref{3.6}) implies
\begin{align}\label{l4-2}
\|\nabla\rho\|_{L^4}^2+\tilde{\rho}^{\gamma-\alpha}\int_0^T\|\nabla\rho\|_{L^4}^2\dif t
\leq~& \bigg(\|\nabla\rho_0\|_{L^4}^2+C\tilde{\rho}^{\frac{15}{7}-\frac{1}{7}\gamma+\frac{13}{7}\delta}+C\tilde{\rho}^2+C\tilde{\rho}^{\frac{45}{16}-\frac{13}{16}\gamma+\frac{5}{4}\delta}+C\tilde{\rho}^{\frac{13}{4}-2\alpha+\frac{3}{4}\gamma+4\delta}\nonumber\\
&+C\tilde{\rho}^{\frac{7}{4}-\frac{1}{2}\alpha-\gamma+\frac{1}{4}\delta}\bigg)\exp\left\{\int_{0}^{T}\|\nabla H\|_{L^2}^\frac{1}{2}\|\nabla^2 H\|_{L^{2}}^\frac{3}{2}\dif t\right\}\nonumber\\
\leq~&\tilde{\rho}^{2+\delta},
\end{align}
provided that $\alpha>\frac{3}{8}\gamma+\frac{5}{8}$. Combining \eqref{l2-2} and \eqref{l4-2}, 
(\ref{2.24}) can be proved.

Moreover, from (\ref{3.1}), (\ref{2.24}) and interpolation inequality, we have
$$\|\rho-\tilde{\rho}\|_{L^\infty}\leq C\|\rho-\tilde{\rho}\|_{L^2}^{\frac{1}{7}}\|\nabla\rho\|_{L^4}^{\frac{6}{7}}\leq C_7\tilde{\rho}^{\frac{15-\gamma+6\delta}{14}}\leq\frac{\tilde{\rho}}{3},$$
provided that $\tilde{\rho}\geq M_{42}\triangleq(3C_7)^{\frac{14}{\gamma-1-6\delta}}$. Taking $\tilde{\rho}\geq M_4\triangleq\max\{M_{41}, M_{42}\}$, we complete the proof.
\end{proof}

Finally, we take $M=\max\{M_1, M_2, M_3, M_4\}$. Noting \eqref{3.09} and Lemma \ref{lem:3.4}--\ref{D:rho}, we can use a standard continuation argument to show that the local solution can be extended to be a global one. This completes the proof of Theorem \ref{main Thm}.



\section*{Acknowledgments} 
The research of this work was supported in part by the National Natural Science Foundation
of China under grants 12371221, 12301277, 12161141004 and 11831011. This work was also partially supported by
the Fundamental Research Funds for the Central Universities and Shanghai Frontiers Science
Center of Modern Analysis.


\begin{thebibliography}{10}

\bibitem{MR4580966}
Y.~Chen, B.~Huang, Y.~Peng, and X.~Shi.
\newblock Global strong solutions to the compressible magnetohydrodynamic
  equations with slip boundary conditions in 3{D} bounded domains.
\newblock {\em J. Differential Equations}, 365:274--325, 2023.

\bibitem{MR2646819}
X.~Hu and D.~Wang.
\newblock Global existence and large-time behavior of solutions to the
  three-dimensional equations of compressible magnetohydrodynamic flows.
\newblock {\em Arch. Ration. Mech. Anal.}, 197(1):203--238, 2010.

\bibitem{MR0839315}
S.~Kawashima.
\newblock Smooth global solutions for two-dimensional equations of
  electromagnetofluid dynamics.
\newblock {\em Japan J. Appl. Math.}, 1(1):207--222, 1984.

\bibitem{1984Systems}
S.~Kawashima.
\newblock Systems of a hyperbolic-parabolic composite type, with applications
  to the equations of magnetohydrodynamics.
\newblock {\em Doctoral Thesis Kyoto Univ}, 1984.

\bibitem{MR3056749}
H.-L. Li, X.~Xu, and J.~Zhang.
\newblock Global classical solutions to 3{D} compressible magnetohydrodynamic
  equations with large oscillations and vacuum.
\newblock {\em SIAM J. Math. Anal.}, 45(3):1356--1387, 2013.

\bibitem{MR3528824}
B.~Lv, X.~Shi, and X.~Xu.
\newblock Global existence and large-time asymptotic behavior of strong
  solutions to the compressible magnetohydrodynamic equations with vacuum.
\newblock {\em Indiana Univ. Math. J.}, 65(3):925--975, 2016.

\bibitem{MR3317636}
Y.~Mei.
\newblock Global classical solutions to the 2{D} compressible {MHD} equations
  with large data and vacuum.
\newblock {\em J. Differential Equations}, 258(9):3304--3359, 2015.

\bibitem{MR2927617}
A.~Suen and D.~Hoff.
\newblock Global low-energy weak solutions of the equations of
  three-dimensional compressible magnetohydrodynamics.
\newblock {\em Arch. Ration. Mech. Anal.}, 205(1):27--58, 2012.

\bibitem{MR390528}
A.~I. Vol\'pert and S.~I. Hudjaev.
\newblock The {C}auchy problem for composite systems of nonlinear differential
  equations.
\newblock {\em Mat. Sb. $($N.S.$)$}, 87(129):504--528, 1972.

\bibitem{MR3620698}
J.~Wu and Y.~Wu.
\newblock Global small solutions to the compressible 2{D} magnetohydrodynamic
  system without magnetic diffusion.
\newblock {\em Adv. Math.}, 310:759--888, 2017.

\bibitem{Yummas}
H.~Yu.
\newblock Global existence of strong solutions to the 3{D} isentropic
  compressible {N}avier-{S}tokes equations with density-dependent viscosities.
\newblock {\em Math. Methods Appl. Sci.}, 46(9):10123--10136, 2023.

\end{thebibliography}
\end{document}